\documentclass[a4paper]{amsart}
\usepackage{amsmath,amssymb,amsfonts}
\usepackage{mathrsfs,latexsym,amsthm,enumerate}
\usepackage{amscd}
\usepackage[all]{xy}

\renewcommand{\d}{\mathbf{d}}

\newtheorem{theorem}{Theorem}[section]
\newtheorem{lemma}[theorem]{Lemma}
\newtheorem{corollary}[theorem]{Corollary}

\newtheorem{proposition}[theorem]{Proposition}

\title{The classifying space of an inverse semigroup}

\author{Ganna Kudryavtseva}
\address{Ganna  Kudryavtseva, University of Ljubljana,
Faculty of Computer and Information Science, 
Tr\v{z}a\v{s}ka cesta 25,
SI-1001, Ljubljana,
 SLOVENIA}
\email{ganna.kudryavtseva\symbol{64}fri.uni-lj.si}
\author{Mark V. Lawson}
\address{Mark V.Lawson, Department of Mathematics
and the
Maxwell Institute for Mathematical Sciences, 
Heriot-Watt University,
Riccarton,
Edinburgh~EH14~4AS,
SCOTLAND}
\email{markl@ma.hw.ac.uk}

\thanks{The first author was partially supported by the ARRS grant P1-0288, and the second was partially supported by the EPSRC grant EP/I033203/1}
 
\begin{document}

\begin{abstract}
We refine Funk's description of the classifying space of an inverse semigroup by replacing his $\ast$-semigroups
by right generalized inverse $\ast$-semigroups.
Our proof uses the idea that presheaves of sets over meet semilattices may be characterized algebraically as right normal bands.
\end{abstract}

\maketitle
\section{Statement of the theorem}

With each inverse semigroup $S$, we shall associate two categories,
the aim of this paper being to prove that these two categories are equivalent.

To define the first, we need the concept of an \'etale action of an inverse semigroup.
These were first explicitly defined in \cite{FS}, but their origins lie in  \cite{Lausch,Log}
and they played an important role in \cite{LS}.
Let $X$ be a non-empty set. 
A {\em left $S$-action} of $S$ on $X$ is a function $S\times X\to X$, defined by $(s,x)\mapsto s\cdot x$ (or $sx$),
such that $(st)x=s(tx)$ for all $s,t\in S$ and $x\in X$.   
If $S$ acts on $X$ we say that $X$ is an {\em $S$-set}.
In this paper, all actions will be assumed left actions.
A {\em left \'{e}tale action} $(S,X,p)$ of $S$ on $X$ is defined as follows \cite{FS,S}. 
Let $E(S)$ denote the semilattice of idempotents of $S$. There is a function $p:X\to E(S)$ 
and a left action $S\times X\to X$ such that the following two conditions hold:
\begin{enumerate}[(E1)]
\item $p(x) \cdot x=x$;
\item $p(s \cdot x)=sp(x)s^{-1}$.
\end{enumerate}
The set $X$ is also partially ordered when we define $x \leq y$ when there exists $e \in E(S)$ such that
$x = e \cdot y$.
A {\em morphism} $\varphi \colon (S,X,p)\to (S,Y,q)$ of left \'{e}tale  actions is a map $\varphi \colon X\to Y$ such that $q(\varphi(x))=p(x)$ 
for any $x\in X$ and $\varphi(s\cdot x)=s\cdot \varphi(x)$ for any $s\in S$ and $x\in X$.
The category of all left \'etale $S$-actions is called the {\em classifying space} or {\em classifying topos} of $S$ and is denoted by $\mathscr{B}(S)$.
This space is the subject of Funk's paper \cite{F}.

In the last section, we shall need a more general notion of morphism.
Let $(S,X,p)$ and $(T,Y,q)$ be \'etale actions where we do not assume that $S$ and $T$ are the same.
Then $(\alpha,\beta) \colon (S,X,p) \rightarrow (T,Y,q)$ is called a {\em morphism}
if $\alpha \colon S \rightarrow T$ is a semigroup homomorphism,
$\beta \colon X \rightarrow Y$ is a function such that $q(\beta (x)) = p(x)$,
and $\beta (s \cdot x) = \alpha (s) \cdot \beta (x)$.

To define our second category, we need some definitions from semigroup theory.
An element $s$ of a semigroup $S$ is said to be {\em (von Neumann) regular} if there is an element $t$,
called an {\em inverse} of $s$, such that $s = sts$ and $t = tst$.
The set of inverses of the element $s$ is denoted by $V(s)$.
In an inverse semigroup $S$, we write the unique inverse of $s$ as $s^{-1}$ and we define
$\mathbf{d}(s) = s^{-1}s$ and $\mathbf{r}(s) = ss^{-1}$. 
A {\em band} is a semigroup in which every element is idempotent
and a {\em right normal band} is a band satisfying the identity $efg = feg$.
A {\em right generalized inverse semigroup} is a regular semigroup whose set of idempotents is a right normal band.
On a regular semigroup $S$, we may define a relation $a \leq b$ if and only if $a = eb = bf$ for some idempotents $e$ and $f$.
This is a partial order called the {\em natural partial order}.
The order need not be compatible with the multiplication but it is precisely when the semigroup $S$
is {\em locally inverse} meaning that each local submonoid, $eSe$, where $e$ is an idempotent, is inverse.
If $S$ is a band, the order is the usual order on idempotents: $e \leq f$ if and only if $e = ef = fe$.
For right generalized inverse semigroups in general, and right normal bands in particular,
this partial order is compatible with the multiplication.

The following definition is taken from \cite{F}, except for (S4) which is new.
A semigroup $S$ is said to be a {\em right $\ast$-semigroup} if it is equipped with a unary operation $s \mapsto s^{\ast}$
satisfying the following axioms:
\begin{enumerate}[({S}1)]
\item\label{s1} $(s^{\ast})^{\ast} = s$. 
\item\label{s2} $s^{\ast} \in V(s)$.
\item\label{s3} $(st)^{\ast} = t^{\ast}(stt^{\ast})^{\ast}$.
\item\label{s4} If $e^{2} = e$ then $e^{\ast} = e$.
\end{enumerate}
Left $\ast$-semigroups can be defined in an analogous way.
In this paper, we shall only be interested in right $\ast$-semigroups and so we shall omit the word
{\em right} in what follows.
Clearly, these semigroups are regular.
Inverse semigroups are special examples where the $\ast$ is just inversion.
We shall be interested in {\em right generalized inverse} $\ast$-semigroups.
{\em Homomorphisms} of $\ast$-semigroups are defined in the obvious way and will sometimes be called {\em $\ast$-homomorphisms}.
A semigroup homomorphism from a $\ast$-semigroup to an inverse semigroup automatically preserves the $\ast$ operation.

It is worth mentioning that axioms (S2) and (S3) arise in a completely different context in the work of Tom Blyth \cite{B1,B2}.

Let $T$ be a right generalized inverse $\ast$-semigroup and $S$ an inverse semigroup.
A semigroup homomorphism $\theta \colon T \rightarrow S$ is said to be {\em \'etale} \cite{F}
if for each $e \in E(T)$ the restriction map $(\theta \mid Te) \colon Te \rightarrow S \theta (e)$ is a bijection.
We denote by $\mbox{Et}/S$ the category of right generalized inverse $\ast$-semigroups \'etale over $S$;
the objects of this category are \'etale homomorphisms $\phi \colon T \rightarrow S$,
and a morphism from $\phi_{1}$ to $\phi_{2}$ is a homomorphism $\theta \colon T_{1} \rightarrow T_{2}$
satisfying $\phi_{1} \theta = \phi_{2}$.
The theorem we shall prove is the following.

\begin{theorem}\label{th: classifying_space} Let $S$ be an inverse semigroup.
Then the classifying space of $S$ is equivalent to the category of right generalized inverse $\ast$-semigroups \'etale over $S$.
\end{theorem} 

The key idea lying behind the work of this paper can be traced back to Wagner \cite{W}
and it is that presheaves of sets over meet semilattices can be regarded as right normal bands.
This idea is explored in more detail in \cite{KL}.
For results on general semigroup theory see \cite{H} and for inverse semigroups \cite{L}.

\section{Proof of the theorem}

A regular semigroup is {\em orthodox} if its set of idempotents forms a band.
On an orthodox semigroup $S$, the relation $\gamma$ defined by
$$s \, \gamma \, t \Leftrightarrow V(s) \cap V(t) \neq \emptyset \Leftrightarrow V(s) = V(t)$$ 
is the minimum inverse congruence.
As usual, we denote Green's relations on any semigroup by $\mathscr{L},\mathscr{R},\mathscr{H},\mathscr{D}$ and $\mathscr{J}$.
The $\mathscr{L}$-class containing the element $a$ is traditionally denoted $L_{a}$.
Right generalized inverse semigroups have a right normal band of idempotents.
We may deduce from this that in such a semigroup $efa = fea$ for any idempotents $e$ and $f$ and any element $a$.

An important property of right generalized inverse semigroups is described below.
It is the beginning of the process of characterizing \'etale maps.

\begin{lemma}\label{le: property} Let $S$ be a right generalized inverse semigroup.
\begin{enumerate}

\item If $a,b \in Se$, where $e$ is an idempotent, and $\gamma (a) = \gamma (b)$ then $a = b$.

\item If $a \, \gamma \, a^{2}$ then $a = a^{2}$. Thus $\gamma$ is {\em idempotent pure}.

\item Let $\gamma (a)\gamma (e) = \gamma (a)$, where $\gamma (e)$ is an idempotent.
Then there exists $b \in Se$ such that $\gamma (b) = \gamma (a)$.

\item The natural homomorphism $S \rightarrow S/\gamma$ is \'etale when $S$ is a right generalized inverse $\ast$-semigroup.

\item Let $\ast$ be a unary operation on $S$ that satisfies (S1), (S2) and (S3).
Then (S4) holds.

\end{enumerate}

\end{lemma}
\begin{proof} (1). Let $a' \in V(a)$.
Then from $ae = a$ we get that $a'ae = a'a$.
It follows that $ea'a \leq e$ and $ea'a \, \mathscr{L} \, a'a$.
But in a right normal band the $\mathscr{L}$-relation is equality.
It follows that $a'a = ea'a$ and so $a'a \leq e$.
If $b' \in V(b)$, we may similarly deduce that $b'b \leq e$.
Right generalized inverse semigroups are locally inverse and so $a'ab'b = b'ba'a$.
We now use the fact that $V(a) = V(b)$.
We therefore have 
$$a'b = a'aa' \cdot b = a'a \cdot a'b = a'b \cdot a'a = a'ba' \cdot a = a'a.$$
It follows that
$$a = a \cdot a'a \cdot a'a = a \cdot a'b \cdot a'b = aa' \cdot ba' \cdot b = ba' \cdot aa' \cdot b = ba' \cdot ba' \cdot b = b$$
as required.

(2). The elements $a$ and $a^{2}$ have the same image under $\gamma$ and $a,a^{2} \in Sa'a$ where $a' \in V(a)$.
It follows by (1) that $a = a^{2}$.

(3). By (2) above we know that $e$ is an idempotent. It is therefore enough simply to define $b = ae$.

(4). This is immediate by (1) and (3) above and the definition of \'etale.

(5). From (S1), (S2) and the fact that $\gamma$ is a homomorphism, it follows that $\gamma (s^{\ast}) = \gamma (s)^{-1}$
for any $s \in S$.
Let $e \in E(S)$.
Then $\gamma (e^{\ast}) = \gamma (e)$.
It follows that $e^{\ast}$ is an idempotent since $\gamma$ is idempotent pure by (2) above.
Since $e \, \gamma \, e^{\ast}$ and we are in a right normal band, we have that $e^{\ast}e = e$.
Applying (S3), we obtain
$$e^{\ast} = (e^{\ast}e)^{\ast} = e^{\ast}(e^{\ast}ee^{\ast})^{\ast} = e^{\ast}e = e.$$
It follows that (S4) holds.
\end{proof}

We shall now describe the form taken by the natural partial order on a regular semigroup in the case important to us.

\begin{lemma}\label{le: order} Let $S$ be a right generalized inverse $\ast$-semigroup.
\begin{enumerate}

\item $a \leq b$ if and only if $a^{\ast} \leq b^{\ast}$.

\item $a \leq b$ if and only if $a = aa^{\ast}b$ if and only if $a^{\ast} = a^{\ast}ab^{\ast}$.

\end{enumerate}
\end{lemma}
\begin{proof} (1). Let $a \leq b$. 
By definition, we have that $a = eb = bf$ for some idempotents $e$ and $f$.
By (S3), we have that $a^{\ast} = (eb)^{\ast} = b^{\ast}(ebb^{\ast})^{\ast}$.
But $ebb^{\ast}$ is an idempotent, and so by (S4) we have that $a^{\ast} = b^{\ast}(ebb^{\ast}) = b^{\ast}e'$ where $e'$ is an idempotent.
We shall now prove that $a^{\ast} = b^{\ast}bfb^{\ast}$.
This follows from Lemma~\ref{le: property}
using the fact that $a^{\ast} \in Sbb^{\ast}$ and $\gamma (b^{\ast}bfb^{\ast}) = \gamma (a)^{-1}$.
The proof of the converse is immediate.

(2) Suppose that $a \leq b$. 
From $a \leq b$ we have that $a = eb = bf$ for some idempotents $e$ and $f$.
But then $ea = a$ and so $eaa^{\ast} = aa^{\ast}$.
It follows that $a = aa^{\ast}eb = eaa^{\ast}b = aa^{\ast}b$ using the fact that the idempotents of $S$ form a right normal band.
Conversely, suppose that $a = aa^{\ast}b$.
Then $\gamma (a) \leq \gamma (b)$ and so $\gamma (a) = \gamma (ba^{\ast}a)$.
But $a,ba^{\ast}a \in Sa^{\ast}a$.
It follows by Lemma~\ref{le: property} that $a = ba^{\ast}a$ and so $a \leq b$, as required.
The proof of the other equivalence is now immediate by this result and (1).
\end{proof}

It follows that the natural partial order on a right generalized inverse $\ast$-semigroup
coincides with the order studied in \cite{F}.

Let $\theta \colon S \rightarrow T$ be a surjective homomorphism of regular semigroups.
We say that it is an {\em $\mathscr{L}$-cover} if for each idempotent $e \in S$ the map
$(\theta \mid L_{e}) \colon L_{e} \rightarrow L_{\theta (e)}$ is bijective.
We could prove some of the results that follow in greater generality.

\begin{proposition}\label{prop: co-ordinatization} Let $S$ be a right generalized inverse semigroup.
\begin{enumerate}

\item The natural map $S \rightarrow S/\gamma$ is an {\em $\mathscr{L}$-cover}.

\item There is a bijection between $S$ and the subset of $S/\gamma \times E(S)$ consisting of those pairs
$(\gamma (s),e)$ where $s's \, \gamma \,  e$ and $s' \in V(s)$.

\end{enumerate}
\end{proposition}
\begin{proof} (1)
Suppose first that $s \,\mathscr{L}\, t$ and $\gamma (s) = \gamma (t)$.
Let $s' \in V(s)$.
Then $s \, \mathscr{L}\, s's$
and
and 
$t \mathscr{L} s't$
since $s' \in V(t)$.
It follows that $s's \, \mathscr{L} \, s't$.
But in a right normal band the $\mathscr{L}$-relation is just equality and so 
$s's = s't$.
We have shown that $s,t \in Ss's$.
Thus by Lemma~\ref{le: property}, we have that $s = t$, as required.

Next, let $e \in E(S)$ and $\gamma (t) \, \mathscr{L} \, \gamma (e)$. 
Let $t' \in V(t)$.
Then $\gamma (t't) \, \mathscr{L} \, \gamma (e)$.
Since both elements are idempotent, we have that $t't \, \gamma\, e$.
It follows that $e = et'te$ and $t't = t'tet't$.
Consider the element $te \in S$.
Then $\gamma (te) = \gamma (t) \gamma (e) = \gamma (t)$.
From $e = (et') te$ it is immediate that $te \, \mathscr{L} \, e$.

(2) Put $S/\gamma \ast E(S)$ equal to the set of ordered pairs satisfying the condition.
Define $\kappa \colon S \rightarrow S/\gamma \ast E(S)$ by $\kappa (s) = (\gamma (s), s's)$.
This is well-defined since in a right normal band the $\mathscr{L}$-relation is equality.
The fact that $\kappa$ is a bijection follows by (1) above.
\end{proof}

In our next result, we characterize \'etale homomorphisms.

\begin{proposition}\label{prop: anja} Let $S$ be an inverse semigroup and $T$ a right generalized inverse $\ast$-semigroup.
\begin{enumerate}

\item Let  $\theta \colon T \rightarrow S$ be an \'etale homomorphism.
Then the image of $\theta$ is a left ideal of $S$.

\item Let $\theta \colon T \rightarrow S$ be a homomorphism 
such that whenever $a,b \in Te$, where $e$ is an idempotent, and $\theta (a) = \theta (b)$ then $a = b$.
Then $\mbox{\rm ker}(\theta) = \gamma$.

\item Let $\theta \colon T \rightarrow S$ be a homomorphism whose kernel is $\gamma$ and whose image is a left ideal of $S$.
Then $\theta$ is \'etale.

\end{enumerate}
\end{proposition}
\begin{proof} (1). Let $x \in \mbox{\rm im}(\theta)$ and $s \in S$.
We prove that $sx \in \mbox{\rm im}(\theta)$. 
Let $x = \theta (t)$ where $t \in T$.
Put $e = t^{\ast}t$.
Then by assumption, $(\theta \mid Te) \colon Te \rightarrow S \theta (e)$ is a bijection.
But $x = x\theta (e)$ and so $sx \in S \theta (e)$.
It follows that there is a $u \in T$ such that $\theta (u) = sx$, as required.

(2). Let $\gamma (a) = \gamma (b)$.
Then $a^{\ast} \in V(b)$.
Thus $b = ba^{\ast}b$ and $a^{\ast} = a^{\ast}ba^{\ast}$.
It follows that $\theta (a)^{-1} = \theta (b)^{-1}$ and so $\theta (a) = \theta (b)$.
Conversely, let $\theta (a) = \theta (b)$.
Then $\theta (a^{\ast}) = \theta (a^{\ast}) \theta(b) \theta (a^{\ast})$ and $\theta (b) = \theta (b) \theta (a^{\ast}) \theta (b)$.
Thus by assumption, $a^{\ast} = a^{\ast}ba^{\ast}$ and $b = ba^{\ast}b$.
We have shown that $V(a) \cap V(b) \neq \emptyset$ and so $\gamma (a) = \gamma (b)$.

(3). Let $e \in E(T)$.
We need to prove that $(\theta \mid Te) \colon Te \rightarrow S \theta (e)$ is a bijection.
Let $x \in S \theta (e)$.
Then $x = x \theta (e)$.
But by assumption, the image of $\theta$ is a left ideal and so $x \in \mbox{\rm im}(\theta)$.
Thus there exists $u \in T$ such that $\theta (u) = x$.
Now observe that $\theta (ue) = x$.
It follows that $(\theta \mid Te)$ is surjective.
Now let $t_{1},t_{2} \in Te$ such that $\theta (t_{1}) = \theta (t_{2})$.
We prove that $t_{1} = t_{2}$.
Observe that 
both $t_{1}^{\ast}t_{1}$ and $t_{1}^{\ast}t_{2}$ are idempotents less than or equal to $e$ and so commute.
But $\theta (t_{1}t_{1}^{\ast}) = \theta (t_{2}t_{1}^{\ast})$ and so $t_{1}t_{1}^{\ast} = t_{2}t_{1}^{\ast}$
by Lemma~2.1(1).
It follows that
$$t_{1}^{\ast}t_{1} = t_{1}^{\ast} \cdot t_{1}t_{1}^{\ast} \cdot t_{1} = t_{1}^{\ast} \cdot t_{2}t_{1}^{\ast} \cdot t_{1}
= t_{1}^{\ast}t_{1} \cdot t_{1}^{\ast}t_{2} = t_{1}^{\ast}t_{2}.$$
Therefore
$$t_{1} 
= t_{1}t_{1}^{\ast}t_{1} 
= t_{1}t_{1}^{\ast}t_{2} 
= t_{2}t_{1}^{\ast}t_{2} 
= t_{2}.$$
\end{proof}

We now begin the proof of our main theorem.
Let $(S,X,p)$ be a left \'etale action.
Define the set
$$S \ast X = \{(s,x) \in S \times X \colon \mathbf{d}(s) = p(x) \}.$$

\begin{proposition}\label{prop: action_to_map} 
Let $(S,X,p)$ be a left \'etale action.
On the set $S \ast X$ define
$$(s,x)(t,y) = (st,\mathbf{d}(st) \cdot y)$$
and denote by $\pi_{X} \colon S \ast X \rightarrow S$ the projection map $(s,x) \mapsto s$.
Define $(s,x)^{\ast} = (s^{-1}, s \cdot x)$.

\begin{enumerate}

\item $S \ast X$ is a right generalized inverse $\ast$-semigroup whose idempotents are precisely the elements of the form $(p(x),x)$.

\item On the regular semigroup $S \ast X$ the natural partial order $(s,x) \leq (t,y)$ is given by $s \leq t$ and $x \leq y$.

\item The projection map $\pi_{X}$ is \'etale. 
It is surjective if and only if the action has global support.

\item Let $\theta \colon X \rightarrow Y$ be a morphism of \'etale actions $(S,X,p)$ and $(S,Y,q)$.
Then $\bar{\theta} \colon S \ast X \rightarrow S \ast Y$ defined by $(s,x) \mapsto (s, \theta (x))$ is a homomorphism of $\ast$-semigroups and
$\pi_{Y} \bar{\theta} = \pi_{X}$. 

\item We have constructed a functor from $\mathscr{B}(S)$ to $\mbox{\rm Et}/S$.

\end{enumerate}
\end{proposition}
 \begin{proof} (1) The proof of associativity is pleasantly trivial.
Observe that elements of the form $(p(x),x)$ are idempotents since
$$(p(x),x)(p(x),x) = (p(x), p(x) \cdot x) = (p(x), x).$$
Conversely, if $(s,x)$ is an idempotent then we have that $s = s^{2}$ and $x = \mathbf{d}(s) \cdot x$.
It follows that $s = e$ is an idempotent and that $e = p(x)$.
We put $(s,x)^{\ast} = (s^{-1}, s \cdot x)$
and this is well-defined since $p(s \cdot x) = sp(x)s^{-1} = ss^{-1}ss^{-1} = \d (s^{-1})$.
It is routine to check that our axioms for a $\ast$-semigroup hold.
A simple calculation shows that
$(p(x),x)(p(y),y) = (p(x)p(y), p(x) \cdot y)$.
It readily follows that the set of idempotents forms a right normal band.
We have therefore shown that $S \ast X$ is a right generalized inverse $\ast$-semigroup.

(2) Suppose that $(s,x) \leq (t,y)$.
By Lemma~\ref{le: order}, we have that
$$(s,x) = (t,y)(s,x)^{\ast}(s,x) 
\mbox{ and } 
(s,x)^{\ast}(s,x) = (s,x)^{\ast}(s,x) (t,y)^{\ast}(t,y).$$
This quickly reduces to $s \leq t$ and $x \leq y$.  
Conversely, suppose that $s \leq t$ and $x \leq y$.
Then 
$$(s,x)(s,x)^{\ast}(t,y) = (s,x)(s^{-1},s \cdot x)(t,y) = (s, \mathbf{d}(s) \cdot y) = (s,x).$$
Applying Lemma~\ref{le: order}, it follows that $(s,x) \leq (t,y)$.

(3) It is immediate that the projection map is a homomorphism
and that its kernel is $\gamma$.
We show that its image is a left ideal.
Let $s = \pi_{X}(s,x)$ and $t \in S$.
Then $\mathbf{d}(ts) \leq \mathbf{d}(s)$ and therefore $ts = \pi_{X}(ts, \mathbf{d}(ts) \cdot x)$.
We now apply Proposition~\ref{prop: anja}.
The fact that the projection map is surjective if and only if the action has global support is easy to check.

(4) The map $\bar{\theta}$ is well-defined since $q(\theta (x)) = p(x)$.
It is a homomorphism because $\theta (s \cdot x) = s \cdot \theta (x)$.
The proofs of the remaining claims are straightforward.

(5) This is now routine.
\end{proof}

We now construct a functor going in the opposite direction.

\begin{proposition}\label{prop: map_to_action} Let $T$ be a right generalized inverse $\ast$-semigroup, $S$ an inverse semigroup and  
$\theta \colon  T \rightarrow S$ an \'etale homomorphism. 

\begin{enumerate}

\item Define $S \times E(T) \rightarrow E(T)$ by 
$$s \cdot e = tt^{\ast}$$ 
where $t^{\ast}t \leq e$ and $\theta (t) = s \theta (e)$.
Also define $p \colon E(T) \rightarrow E(S)$ by $p(e) = \theta (e)$.
Then $(S,E(T),p)$ is a a left \'etale action.

\item Let $\alpha \colon T_{1} \rightarrow T_{2}$ be a $\ast$-homomorphism of \'etale maps from
$\theta_{1} \colon T_{1} \rightarrow S$ to $\theta_{2} \colon T_{2} \rightarrow S$.
Then $\bar{\alpha} = (\alpha \mid E(T_{1})) \colon E(T_{1}) \rightarrow E(T_{2})$
is a morphism of \'etale actions.

\item We have constructed a functor from $\mbox{\rm Et}/S$ to $\mathscr{B}(S)$.

\end{enumerate}
\end{proposition}
\begin{proof} 
(1) Because the map $\theta$ is \'etale the element $t$ is uniquely defined.
We prove first that $s \cdot (t \cdot e) = (st) \cdot e$.
By definition, $(st) \cdot e = cc^{\ast}$ where $c^{\ast}c \leq e$ and $\theta (c) = st \theta (e)$.
In addition,
$t \cdot e = bb^{\ast}$ where $b^{\ast}b \leq e$ and $\theta (b) = t \theta (e)$,
and
$s \cdot bb^{\ast} = aa^{\ast}$ where $a^{\ast}a \leq bb^{\ast}$ and $\theta (a) = s \theta (bb^{\ast})$.
Observe that
$(ab)^{\ast}ab = b^{\ast}(abb^{\ast})^{\ast}ab = b^{\ast}a^{\ast}ab$ since $a^{\ast}a \leq bb^{\ast}$.
But  $b^{\ast}a^{\ast}ab \leq b^{\ast}b \leq e$.
Thus $(ab)^{\ast}ab \leq e$.
In addition, 
$\theta (ab) = \theta (a) \theta (b) = s \theta (bb^{\ast}) t \theta (e) = st \theta (e)$.
It follows by uniqueness that $c = ab$.
But $cc^{\ast} = ab(ab)^{\ast} = abb^{\ast}(abb^{\ast})^{\ast} = aa^{\ast}$.
Where throughout we have used axiom (S3) for $\ast$-semigroups.
We have an action, we now prove that it is an \'etale action.

(E1) holds: $p(e) \cdot e = aa^{\ast}$ where $a^{\ast}a \leq e$ and $\theta (a) = p(e)\theta (e) = p(e)$.
It follows by uniqueness and axiom (S4) that $a = e$ and so $p(e) \cdot e = e$, as required. 

(E2) holds: $p(s \cdot e) = \theta (s \cdot e)$.
Let $t^{\ast}t \leq e$ and $\theta (t) = sp(e)$.
Now $\theta (tt^{\ast}) = \theta (t) \theta (t)^{-1} = sp(e)s^{-1}$, as required.

(2) Since $\alpha$ is a homomorphism of semigroups $\bar{\alpha}$ is a well-defined map.
We have to show that it is a map of \'etale actions.
Let $s \cdot e = aa^{\ast}$ where $a^{\ast}a \leq e$, and $\theta_{1}(a) = s \theta_{1}(e)$.
Let $s \cdot \bar{\alpha}(e) = bb^{\ast}$ where $b^{\ast}b \leq \bar{\alpha}(e)$, 
and $\theta_{2} (b) = s \theta_{2} (\bar{\alpha}(e)) = s \theta_{1}(e)$.
But $\alpha (a)^{\ast}\alpha (a) \leq \overline{\alpha}(e)$ and $\theta_{2}(\alpha (a)) = \theta_{2} (b)$.
It follows by uniqueness that $\alpha (a) = b$.
Hence $\bar{\alpha}(s \cdot e) = s \cdot \bar{\alpha}(e)$.
It now readily follows that $\bar{\alpha}$ is a morphism of left \'etale actions.

(3) The proof of this is routine.
\end{proof}

It only remains to show that the two functors defined above determine an equivalence of categories.

\begin{proposition}\label{prop: iteration} \mbox{} 
\begin{enumerate}

\item Let $(S,X,p)$ be a left \'etale action.
Then this is isomorphic to the left \'etale action constructed from $\pi_{X} \colon S \ast X \rightarrow S$.  

\item Let $\theta \colon  T \rightarrow S$ be an \'etale homomorphism from a right generalized $\ast$-semigroup to an inverse semigroup.
Then this is isomorphic to the \'etale map constructed from the left \'etale action $(S,E(T),p)$.

\end{enumerate}
\end{proposition}
\begin{proof} 

(1) The set $E(S \ast X)$ is in bijective correspondence with the set $X$ via the map $(p(x),x) \mapsto x$.
By definition $s \cdot (p(x),x) = (t,y)(t,y)^{\ast}$ where
$(t,y)(p(x),x) = (t,y)$ and $\pi_{X}(t,y) = sp(x)$.
It follows that $t = sp(x)$.
Now $(t,y)(t,y)^{\ast} = (tt^{-1}, t \cdot y)$ and $y = \d (t) \cdot x$.
It follows that $(t,y)(t,y)^{\ast} = (p(s \cdot x), s \cdot x)$.
Hence the two \'etale actions are naturally isomorphic.

(2) Given an \'etale map $\theta \colon T \rightarrow S$, we may construct an \'etale action $S \times E(T) \rightarrow E(T)$.
Hence we may construct an \'etale map $\pi_{E(T)} \colon S \ast E(T) \rightarrow S$.
Define $\alpha \colon T \rightarrow S \ast E(T)$ by $\alpha (t) = (\theta (t), t^{\ast}t)$.
This is well-defined and  $\pi_{E(T)} \alpha = \theta$.
Observe that $\alpha$ is a bijection by Proposition~\ref{prop: co-ordinatization}.
It remains to show that it is a $\ast$-homomorphsim which is routine.
\end{proof}

We have therefore proved Theorem~\ref{th: classifying_space}.

We now consider the case of \'etale actions {\em with global support}.
Let $(S,X,p)$ be such an action.
Then $\pi_{X} \colon S \ast X \rightarrow S$ is a surjective \'etale map from the right generalized inverse $\ast$-semigroup
by Proposition~2.5.
But by Proposition~2.4, the kernel of $\phi_{X}$ is $\gamma$.
It follows that the $\ast$-semigroup $S \ast X$ contains all the essential information about the action $(S,X,p)$.

\begin{theorem}\label{the: global_support} 
The category of \'etale actions of inverse semigroups with global support is equivalent to the 
category of right generalized inverse $\ast$-semigroups.
\end{theorem}
\begin{proof} We first define two functors.

Let $\theta \colon S \rightarrow T$ be a homomorphism between right generalized inverse $\ast$-semigroups.
Observe that $s \gamma t$ implies that $\theta (s) \gamma \theta (t)$.
We may therefore define a homomorphism $\theta_{1} \colon S/\gamma \rightarrow T/\gamma$.
There is also a homomorphism $\theta_{2} \colon E(S) \rightarrow E(T)$.
Let $(S/\gamma, E(S), p_{S})$ and $(T/\gamma, E(T),p_{T})$ be \'etale actions with global support associated with $S$ and $T$ respectively.
The fact that $(\theta_{1},\theta_{2})$ is a morphism of \'etale sets follows readly from the definition
of the actions and the fact that $\theta$ is a $\ast$-homomorphism.

Now let $(\alpha, \beta) \colon (S,X,p) \rightarrow (T,Y,q)$ be a morphism of \'etale actions with global support.
Define $\theta \colon S \ast X \rightarrow T \ast Y$ by $\theta (s,x) = (\alpha (s),\beta (x))$.
It is routine to check that this is a well-defined map and a $\ast$-homomorphism.

The fact that these two functors yield an equivalence of categories essentially follows by Proposition~\ref{prop: iteration}.
\end{proof}

Let $T$ be a generalized inverse semigroup.
We say that $T$ is {\em over the inverse semigroup $T/\gamma$.}
The proof of the following is now immediate.

\begin{corollary} Let $S$ be an inverse semigroup.
Then the category of \'etale actions of $S$ with global support is equivalent to the category
of right generealized inverse $\ast$-semigroups over $S$.
\end{corollary}

We now describe an example that provided intuition for our construction.
Let $H$ be a group that acts on a set $X$ on the left. 
Then $G = H \times X$ can be endowed with the structure of a groupoid in a construction that goes back to Ehresmann.
The question arises of how we might capture the action purely algebraically.
We can regard $X$ as a right zero semigroup and then $G$ becomes a right group:  
that is, a direct product of a group and a right zero semigroup. 
This leads to a forgetful functor from left actions of $G$ to right groups with base group $G$. 
If $(h,x)\in G$ then its domain is $x$ via the identification of $(x, id)$ with $x$, and its range is $hx$. 
However, when we pass from the groupoid $G$ to the right group $G$ we lose information about ranges. 
A way to record this information in $G$ is to preserve the inversion of the groupoid by considering $G$ as a $*$-semigroup 
by defining $(h,x)^*=(h,x)^{-1}=(h^{-1}, hx)$. 
Then the range of $(h,x)$ is recorded as the domain of  $(x,h)^*$. 
The axioms of $*$-semigroup given in \cite{F} are enough to prove that left actions of $H$ are the same thing 
as right groups with base groups $G$ that have the structure of $*$-semigroups.
In our theorem above, we are essentially replacing the group by an inductive groupoid.


\end{document}